\documentclass{article}
\usepackage{amssymb,amsmath}
\usepackage{graphicx}

\def\qed{\hfill {\hbox{${\vcenter{\vbox{               
   \hrule height 0.4pt\hbox{\vrule width 0.4pt height 6pt
   \kern5pt\vrule width 0.4pt}\hrule height 0.4pt}}}$}}}

\newtheorem{theorem}{Theorem}
\newtheorem{definition}{Definition}
\newtheorem{lemma}[theorem]{Lemma}

\newtheorem{example}{Example}

\newtheorem{remark}{Remark}

\newenvironment{proof}[1][Proof]{\smallskip\noindent{\bf #1.}\quad}%
{\qed\par\medskip}

\date{}

\title{\Large \textbf{Twisted virtual biracks and their twisted virtual link invariants}}

\author{Jessica Ceniceros\footnote{Email: \texttt{jceniceros11@students.claremontmckenna.edu}}
 \and Sam Nelson\footnote{Email: \texttt{knots@esotericka.org}}}

\begin{document}
\maketitle

\begin{abstract}
A virtual link can be understood as a link in a trivial $I$-bundle over an 
orientable compact surface with genus. A \textit{twisted virtual link} is 
a link in a trivial $I$-bundle over a not-necessarily orientable compact
surface. A \textit{twisted virtual birack} is an algebraic structure with
axioms derived from the twisted virtual Reidemeister moves. We extend 
a method previously used with racks and biracks to the twisted case to define 
computable invariants of twisted virtual links using finite twisted virtual 
biracks with birack rank $N\ge 1$. As an application, we classify twist
structures on the virtual Hopf link.
\end{abstract}

\begin{center}
\parbox{6in}{

}
\end{center}

\section{\large\textbf{Introduction}}

In \cite{K,KK} \textit{virtual knots and links} and equivalently
\textit{abstract knots links} respectively were introduced. Virtual knots 
were initially conceived as combinatorial objects, Reidemeister equivalence 
classes of Gauss codes, where abstract knots were geometric in nature, 
knot diagrams drawn on minimalistic supporting surfaces. In \cite{CKS} a 
geometric interpretation of virtual and abstract links as isotopy classes 
of simple closed curves in $I$-bundles over compact oriented surfaces modulo
stabilization moves was developed.

In recent work such as \cite{B,K2}, virtual and abstract links are extended 
to allow compact non-orientable supporting surfaces; the resulting links are 
called \textit{twisted virtual links}. Invariants of twisted virtual links
such as the twisted Jones polynomial and the twisted knot group have been
introduced and studied. Twisted virtual knot theory is very new, with many 
interesting open questions, such as finding a supporting surface of minimal
genus for a given twisted virtual knot or link.

Quandle- and biquandle-based invariants of twisted virtual links were first
considered in \cite{SPC}. In \cite{N} a counting invariant of unframed
oriented classical and virtual knots and links was defined using labelings 
by finite racks and extended to finite biracks in \cite{N2}. In this paper 
we extend the birack counting invariant to the case of twisted virtual links. 

The paper is organized as follows: in section \ref{TVL} we recall the basics 
of twisted virtual link theory from \cite{B,K2} and make a few observations 
which will be useful in later sections. In section \ref{TVB} we recall
twisted virtual biracks give some examples. In section \ref{INV} we define 
the twisted virtual birack counting invariant and provide examples and 
sample computations of the invariant. As an application, we classify twist
structures on the virtual Hopf link. We end with a few open questions for 
future research in section \ref{Q}.

\section{\large\textbf{Twisted Virtual Links}}\label{TVL}

\textit{Twisted virtual links} were introduced in \cite{B} and subsequently 
studied in works such as \cite{K2,SPC}. Twisted virtual links extend
the concept of \textit{virtual links} from previous work \cite{K,KK};
where virtual links arise by drawing link diagrams on compact orientable
surfaces with nonzero genus, twisted virtual links arise when we draw
link diagrams on compact surfaces allowing nonzero genus and nonzero cross-cap 
number. 

Knot and link diagrams are usually drawn on flat paper without explicitly
specifying a supporting surface $\Sigma$ on which the knot diagram is drawn. 
If we do explicitly draw $\Sigma$, we have a \textit{link-surface} diagram. 
Often we will remove a disk from $\Sigma\setminus L$ so we can flatten 
$\Sigma$ as depicted. Virtual crossings correspond to crossed bands while
classical crossings correspond to crossings drawn on $\Sigma$.

\[\includegraphics{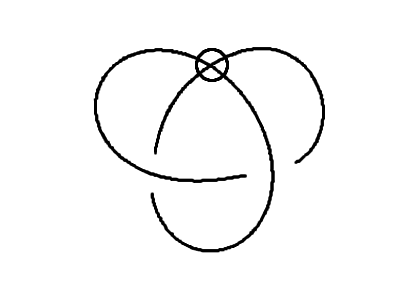}\quad \includegraphics{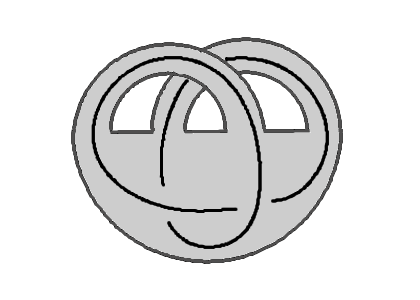}\]

Geometrically, a twisted virtual link is a stable equivalence class of simple 
closed curves in an $I$-bundle, i.e. an ambient space obtained by thickening 
the surface $\Sigma$ on which the link diagram is drawn. Here ``stable 
equivalence''
means that in addition to ambient isotopy of the link within the thickened
surface, we can \textit{stabilize} the surface $\Sigma$ by adding or deleting
torus summands or cross caps not containing the link. If we remove a disk from 
$\Sigma\setminus L$ and flatten the resulting $\Sigma'$ in the usual way to get
\[\includegraphics{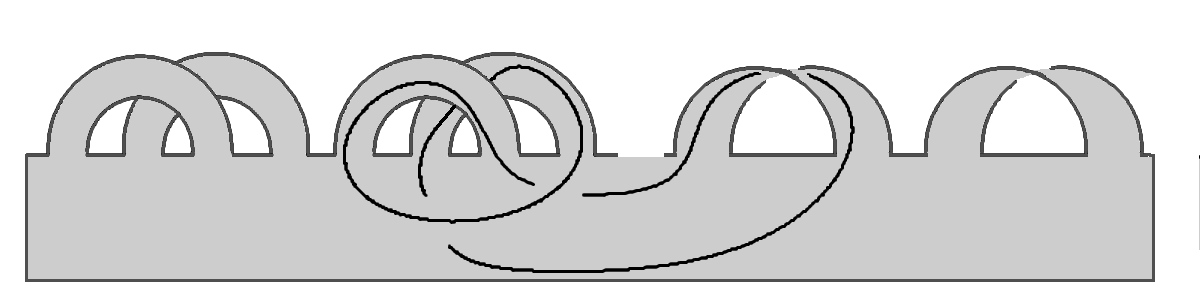}\]
then stabilization moves have the form
\[\includegraphics{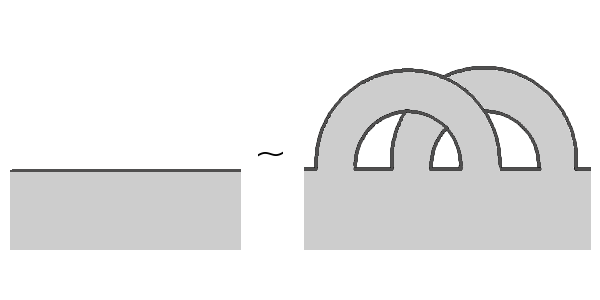} \quad\raisebox{0.5in}{and}\quad 
\includegraphics{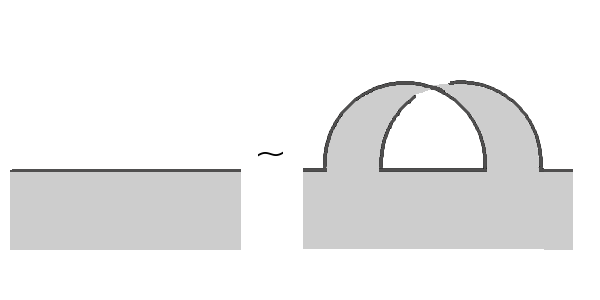}.\]

We can represent twisted virtual links combinatorially without having to draw 
the supporting surface $\Sigma$ by representing
crossings arising from genus in $\Sigma$ with circled self-intersections
known as \textit{virtual crossings} (see \cite{K,KK}) and representing
places where our link traverses a cross cap in $\Sigma$ with a small bar.

\[\includegraphics{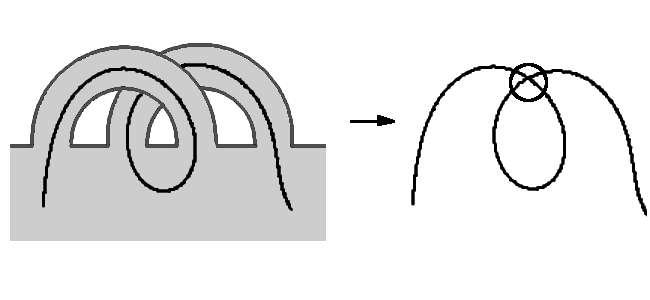}\quad 
\quad \quad \includegraphics{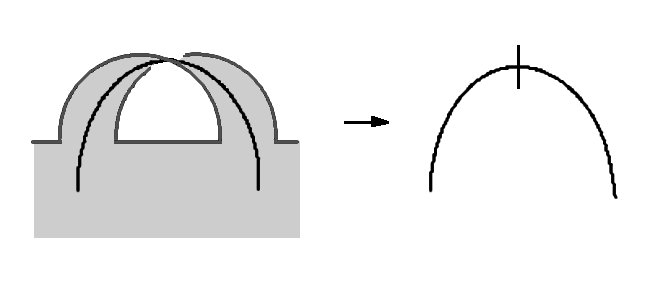}\]

Then for instance the twisted virtual Hopf link diagram below corresponds to 
the link-surface diagram shown.

\[\includegraphics{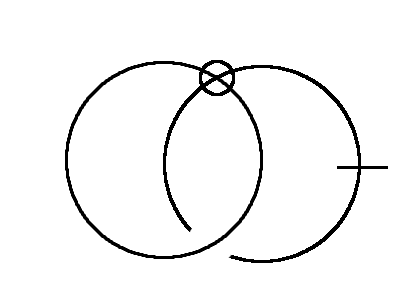}\quad\includegraphics{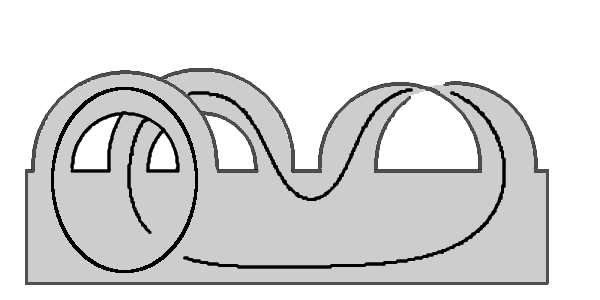}\]

The portions of a twisted virtual link diagram $L$ between overcrossings, 
undercrossings, virtual crossings and bars are \textit{semiarcs}. For instance, 
the twisted virtual Hopf link diagram above has six semiarcs.

In \cite{B} it is shown that stable isotopy of twisted virtual links 
corresponds to the equivalence relation on twisted virtual link diagrams
generated by the \textit{twisted virtual Reidemeister moves}:

\[\includegraphics{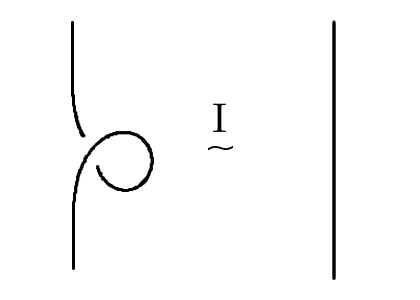}\quad \quad \quad
\includegraphics{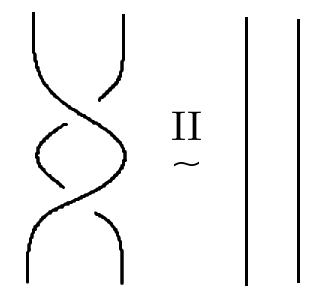} \quad\quad \quad
\includegraphics{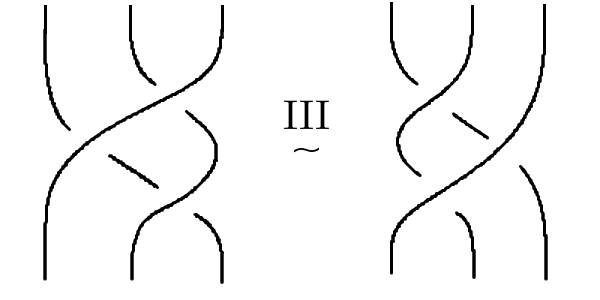} \]
\[\includegraphics{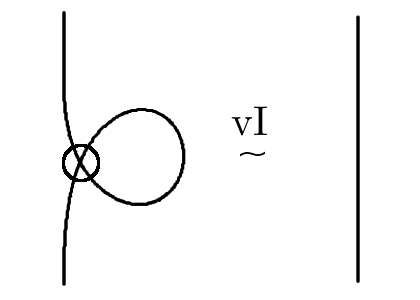}\quad\ \
\includegraphics{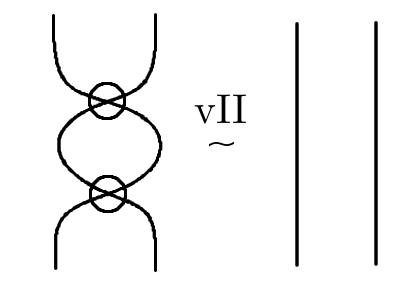} \quad\ \
\includegraphics{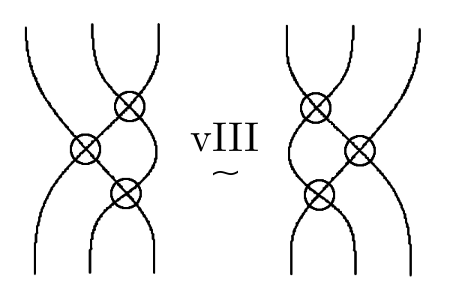} \quad\ \
\includegraphics{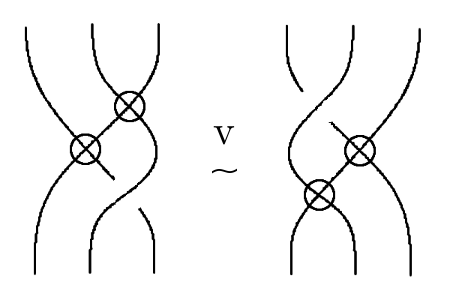} \]
\[\raisebox{0.1in}{\includegraphics{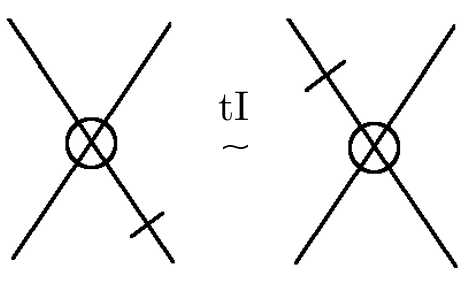}} \quad\quad \quad
\includegraphics{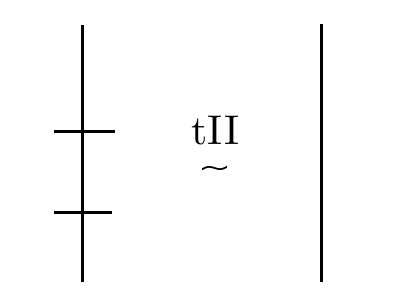}\quad \quad \quad
\includegraphics{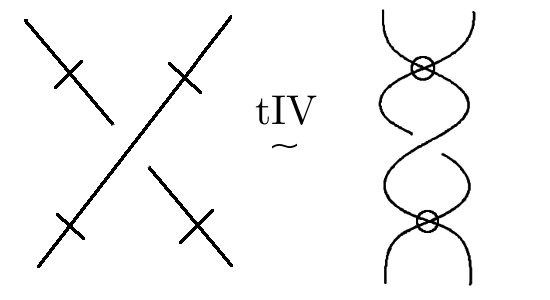} 
\]

Each of these moves can be understood in terms of link-surface diagrams 
or abstract link diagrams; for instance, the last move looks like:
\[\includegraphics{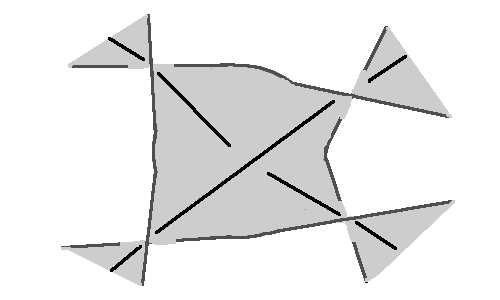}\raisebox{0.5in}{$\sim$}
\raisebox{-0.3in}{\includegraphics{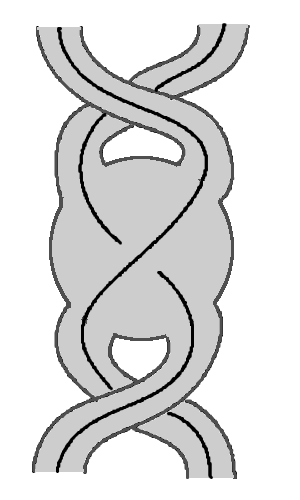}}
\]

Note that we do not need twist bars covering multiple strands since for 
any two neighboring strands going through a twisted band, we 
can remove a disc from the surface between the strands and replace the 
multi-strand bar with two bars and a virtual crossing:
\[\includegraphics{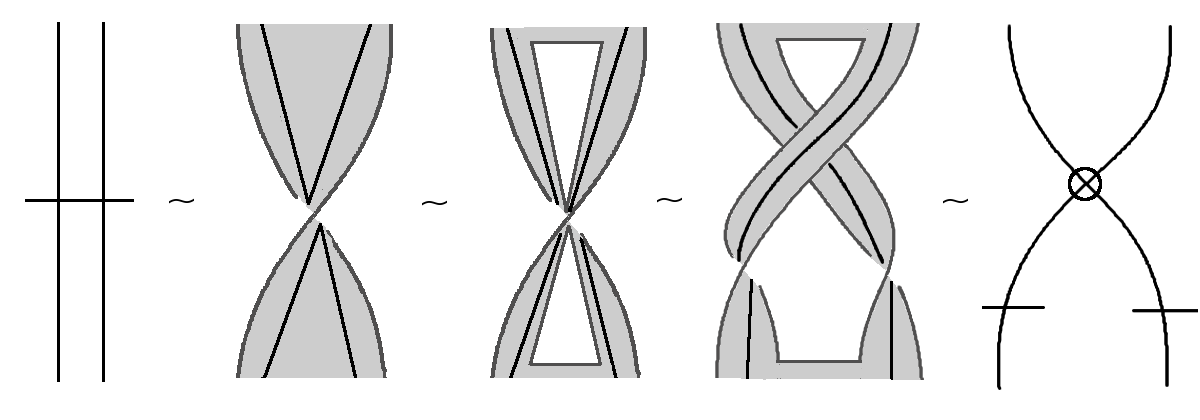}\]

The four virtual moves together imply the \textit{detour move}, which says 
that a strand with only virtual crossings can be moved past any tangle
containing classical or virtual crossings; move tI implies that strands
with only virtual crossings can detour past twist bars as well.
\[\includegraphics{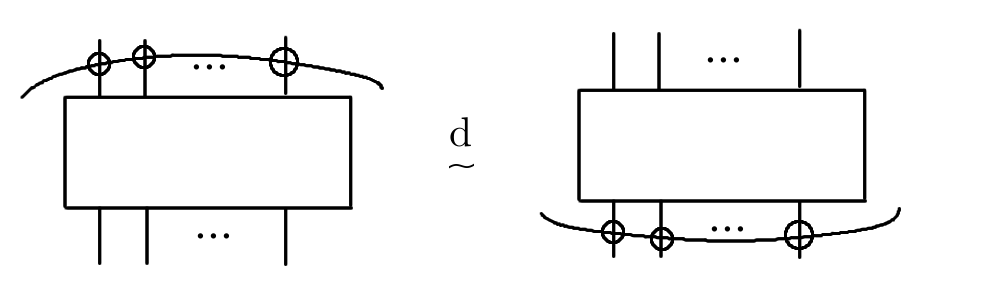}\]

Replacing the usual classical Reidemeister type I move with the blackboard 
framed type I moves 
\[\includegraphics{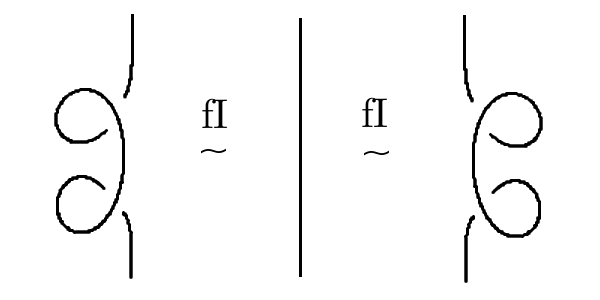}
\quad \raisebox{0.5in}{or equivalently}\quad 
\includegraphics{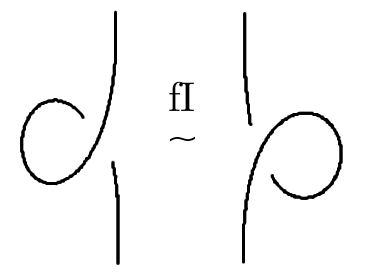} \ \raisebox{0.5in}{and}\
\includegraphics{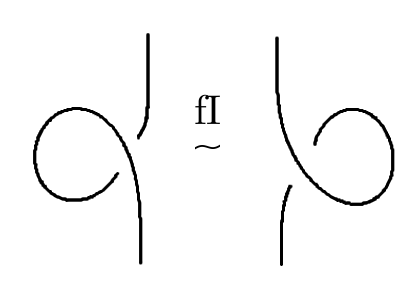}
\] yields 
\textit{blackboard framed twisted virtual isotopy}. Including orientations
on the link components gives \textit{oriented blackboard framed twisted 
virtual isotopy}. 

We will primarily be interested in using invariants
of oriented blackboard framed twisted virtual isotopy to define an
invariant of oriented unframed twisted virtual isotopy analogous to
those defined in \cite{N} and \cite{N2}.

We will find the following observations useful in the next section.

\begin{lemma}
A twist bar can be moved past a classical kink with virtual twisted 
blackboard framed isotopy moves.
\end{lemma}

\begin{proof}
\[\includegraphics{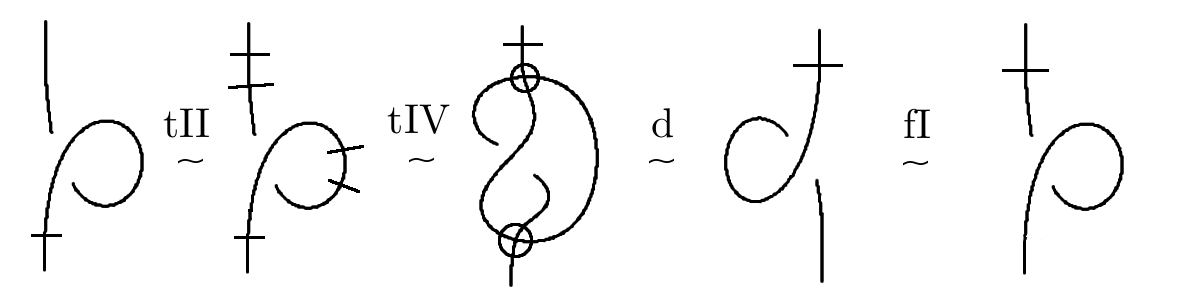} 
\]
\end{proof}

\begin{lemma}
The two oriented versions of the last twisted virtual move are equivalent,
i.e we have
\[\includegraphics{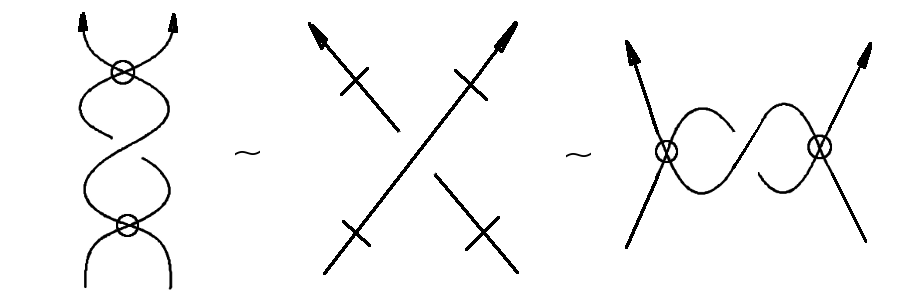}.\]
\end{lemma}

\begin{proof}
\[\includegraphics{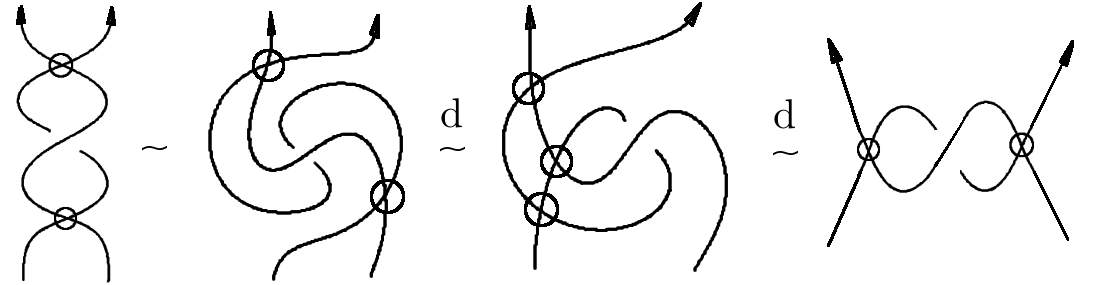} 
\]
\end{proof}

\section{\large\textbf{Twisted Virtual Biracks}}\label{TVB}

We begin with a definition slightly modified from \cite{SPC}.

\begin{definition}
\textup{Let $X$ be a set and $\Delta:X\to X\times X$ the diagonal map
$\Delta(x)=(x,x)$. A \textit{twisted virtual birack} is a set $X$ with 
invertible maps $B,V:X\times X\to X\times X$, 
and an involution $T:X\to X$ satisfying the axioms below where $F_j$ 
denotes a map $F$ followed by projection onto the $j$th component:
\begin{list}{}{}
\item[(i)] $B$ and $V$ are \textit{sideways invertible:} there exist
unique invertible maps $S:X\times X\to X\times X$ and 
$vS:X\times X\to X\times X$ such that for all $x,y\in X$ we have
\[S(B_1(x,y),x)=(B_2(x,y),y)\quad \mathrm{and}\quad
vS(V_1(x,y),x)=(V_2(x,y),y);\]
\item[(ii)] The compositions $(S^{\pm 1}\circ \Delta)_k$ and
$(vS^{\pm 1}\circ \Delta)_k$ are bijections for $k=1,2$;
\item[(iii)] $(vS\circ \Delta)_1=(vS\circ \Delta)_2$;
\item[(iv)] $B$ and $V$ satisfy the set-theoretic Yang-Baxter equations:
\[(B\times \mathrm{Id}_X)(\mathrm{Id}_X\times B)(B\times \mathrm{Id}_X)=
(\mathrm{Id}_X\times B)(B\times \mathrm{Id}_X)(\mathrm{Id}_X\times B),
\]
\[(V\times \mathrm{Id}_X)(\mathrm{Id}_X\times V)(V\times \mathrm{Id}_X)=
(\mathrm{Id}_X\times V)(V\times \mathrm{Id}_X)(\mathrm{Id}_X\times V),
\]
and
\[(B\times \mathrm{Id}_X)(\mathrm{Id}_X\times V)(V\times \mathrm{Id}_X)=
(\mathrm{Id}_X\times V)(V\times \mathrm{Id}_X)(\mathrm{Id}_X\times B);
\]
\item[(v)] \[(T\times \mathrm{Id})V=V(\mathrm{Id}\times T)
\quad \mathrm{and}\quad(\mathrm{Id}\times T)V=V(T\times \mathrm{Id}),
\]
\item[(vi)] \[(T\times T)B(T\times T)=VBV.\]
\end{list}
If we also have 
\begin{itemize}
\item[(iii$'$)] $(S\circ \Delta)_1=(S\circ \Delta)_2$,
\end{itemize} 
then $X$ is a \textit{twisted virtual biquandle}.}
\end{definition}

The twisted virtual birack axioms are obtained from the blackboard framed 
twisted virtual Reidemeister moves using the following semiarc-labeling
scheme:
\[\includegraphics{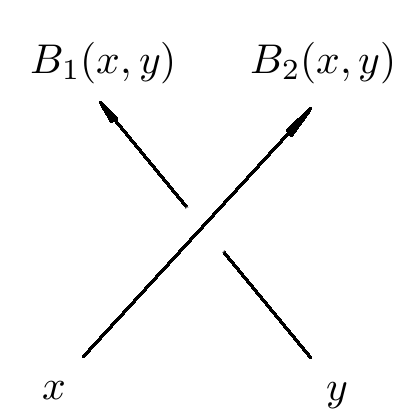}\quad \includegraphics{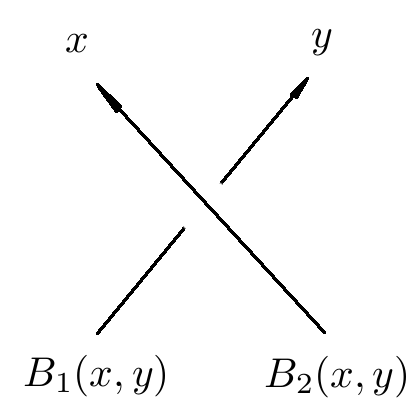}
\quad \includegraphics{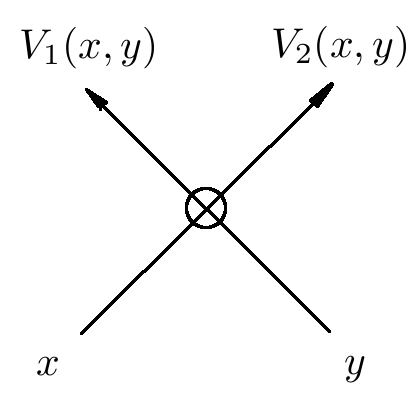}\quad \includegraphics{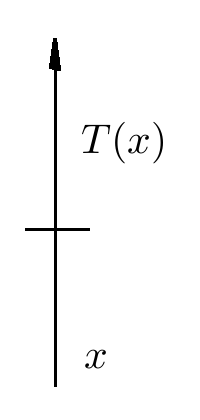}\]
See \cite{SPC} for more details.

Let $X$ and $Y$ be twisted virtual biracks with maps $B_X,V_X,T_X$ and
$B_Y,V_Y,T_Y$ respectively. As with other algebraic structures, we have the 
following common notions:
\begin{itemize}
\item A map $f:X\to Y$ is a \textit{homomorphism of twisted virtual biracks}
if 
\[B_Y(f\times f)=(f\times f) B_X,\quad 
V_Y(f\times f)=(f\times f) V_X\quad  \mathrm{and} \quad 
T_Y f=f T_X;\] and
\item If $Y\subset X$, then $Y$ is a \textit{twisted virtual subbirack} of $X$
provided 
\[B_Y=B_XI,\quad T_Y=T_X I,\quad \mathrm{and} \quad 
T_Y=T_X I\] 
where $I:Y\to X$ is inclusion.
\end{itemize}

The map $\pi=(S^{-1}\circ\Delta)_1\circ(S^{-1}\circ\Delta)_2^{-1}$ represents
going through a positive kink; this is known as the \textit{kink map}. A
birack is \textit{biquandle} if its kink map is the identity; otherwise,
$\pi$ is an element of $S_n$ and its order $N$ is the \textit{birack rank}
or \textit{birack characteristic} of $X$.

\begin{remark}
\textup{If $X$ is a twisted virtual birack, then the map $B$ defines a
birack structure on $X$ and the map $V$ defines a semiquandle structure
on $X$, i.e. a biquandle structure  with $V^2=\mathrm{Id}$. The pair 
$B,V$ then defines a \textit{virtual birack} structure
on $X$, i.e. a birack structure with a compatible semiquandle structure.
Thus, a twisted virtual birack is a virtual birack with a compatible
twist map $T$. See \cite{HN,KM}.}

\textup{Moreover, if the birack rank of $(X,B)$ is 1, then we have a 
\textit{twisted virtual biquandle}. If $B_2(x,y)=x$ for all $x\in X$, then
we have a \textit{twisted virtual rack}, and a twisted virtual rack which is
also a twisted virtual biquandle is a \textit{twisted virtual quandle}.
We summarize the different structures with a table:}

\medskip

\textup{
\begin{tabular}{ll}
\begin{tabular}{|l|l|l|} \hline
& $N=1$ & $N\ne 1$ \\ \hline
$B_2=\mathrm{Id}_X$ & quandle & rack \\ \hline
$B_2\ne\mathrm{Id}_X$ & biquandle & birack \\\hline
\end{tabular} &
\begin{tabular}{cl}
$\bullet$ & If $B^2=\mathrm{Id}_{X\times X}$, replace ``bi--'' with ``semi--'' \\
$\bullet$ & If $V(x,y)\ne (y,x)$, add ``virtual'' \\
$\bullet$ & If $T\ne \mathrm{Id}_X$, add ``twisted'' \\
\end{tabular}
\end{tabular}}

\medskip

\end{remark}

\begin{example}
\textup{Let $X$ be a commutative ring with multiplicative group of units 
$X^{\ast}$, $t,r\in X^{\ast}$ and $s\in X$ satisfying $s^2=(1-tr)s$; then $X$ is 
a birack with $B(x,y)=(ty+sx,rx)$
known as a $(t,s,r)$-birack (see \cite{N}). If we likewise choose 
$v,w\in X^{\ast}$, $u\in X$ satisfying $u^2=(1-vw)u$ and set 
$V(x,y)=(vy+ux,wx)$, then $V^2=\mathrm{Id}$ requires that 
\[x=(vw+u^2)x+uvy\quad \mathrm{and}\quad y=vwy+wux\]
which in turn implies $w=v^{-1}$ and $u=0$. Note that $u^2=0^2=(1-1)0=(1-vw)u$
and the condition $u^2=(1-vw)u$ is automatic in this case.
Thus, our virtual operation becomes
$V(x,y)=(vy,v^{-1}x)$. The mixed virtual move then 
requires that multiplication by $v$ commutes with multiplication by $t,s$ and 
$r$. }

\textup{Choosing $T\in X^{\ast}$ so that $T:X\to X$ is given by $T(x)=Tx$, the 
twisted moves then  require that $T^2=1$, multiplication by $v$ commutes with 
multiplication by $T$, and that
\[T^2rx=rx=v^{-2}tx+sy, \quad \mathrm{and}\quad v^2ry=T^2ty+T^2sx=ty+sx.\]
Comparing coefficients, we see that we need $s=0$ and $v^2r=t$. Since $s=0$
implies $s^2=(1-tr)s$, the conditions on our coefficients reduce to
$t,r,v,T\in X^{\ast}$ with $T^2=1$ and $v^2r=t$. Thus, for any
commutative ring $X$ with units 
$t,r,v,T$ satisfying $v^2r=t$ and $T^2=1$, we have a twisted virtual 
birack structure on $X$ defined by}
\[B(x,y)=(ty,rx), \quad V(x,y)=(vy, v^{-1}x),\quad \mathrm{and}\quad T(x)=Tx.\]
\end{example}

\begin{remark}\label{noinc}
\textup{It is curious that unlike the case of virtual biracks, 
simply making the twist operation trivial does not give a natural
embedding of the category of virtual biracks into the
category of twisted virtual biracks. More precisely, every birack is a
virtual birack with virtual operation $V(x,y)=(y,x)$; on the other
hand, if we have trivial operations at the virtual crossings and 
twist bars, then the twisted move tv requires the component maps of
$B$ to be equal, i.e. that $y_x=y^x$, a condition which is false for 
most biracks. In particular, for a given virtual birack, the set of 
compatible twist structures may be empty. }
\end{remark}

\begin{example}
\textup{We can represent a twisted virtual birack structure on a finite set
$X=\{x_1,\dots,x_n\}$ with an $n\times (4n+1)$-matrix $M_X=[U|L|vU|vL|T]$ 
with $n\times n$ blocks $U,L,vU,vL$ encoding the operations $B$, $V$ and an
$n\times 1$ block $T$ encoding the involution $T$ in the following way:
\begin{itemize}
\item if $B(x_i,x_j)=(x_k,x_l)$ we set $U_{ji}=k$ and $L_{ij}=l$ (note the 
reversed order of the subscripts in $U$; this is for compatibility with 
previous work);
\item if $V(x_i,x_j)=(x_k,x_l)$ we set $vU_{ji}=k$ and v$L_{ij}=l$, and
\item if $T(x_i)=x_j$ we set $T[i,1]=j$.
\end{itemize}
Every finite twisted virtual birack can be encoded by such a matrix, and
conversely an $n\times (4n+1)$ matrix with entries in $\{1,2,\dots, n\}$
defines a twisted virtual birack provided the twisted virtual birack axioms 
are all satisfied by the maps $B,V$ and $T$ defined by the matrix.}

\textup{
For instance, our Python computations reveal that there are
eight twisted virtual birack structures on the set $X=\{x_1,x_2\}$, given by
the  matrices
\[
\left[\begin{array}{cc|cc|cc|cc|c}
2 & 2 & 2 & 2 & 1 & 1 & 1 & 1 & 2 \\
1 & 1 & 1 & 1 & 2 & 2 & 2 & 2 & 1 \\
\end{array}\right],
\left[\begin{array}{cc|cc|cc|cc|c}
2 & 2 & 2 & 2 & 2 & 2 & 2 & 2 & 2 \\
1 & 1 & 1 & 1 & 1 & 1 & 1 & 1 & 1 \\
\end{array}\right], \]\[
\left[\begin{array}{cc|cc|cc|cc|c}
1 & 1 & 1 & 1 & 1 & 1 & 1 & 1 & 2 \\
2 & 2 & 2 & 2 & 2 & 2 & 2 & 2 & 1 \\
\end{array}\right],
\left[\begin{array}{cc|cc|cc|cc|c}
1 & 1 & 1 & 1 & 2 & 2 & 2 & 2 & 2 \\
2 & 2 & 2 & 2 & 1 & 1 & 1 & 1 & 1 \\
\end{array}\right], \]
\[
\left[\begin{array}{cc|cc|cc|cc|c}
2 & 2 & 2 & 2 & 1 & 1 & 1 & 1 & 1 \\
1 & 1 & 1 & 1 & 2 & 2 & 2 & 2 & 2 \\
\end{array}\right],
\left[\begin{array}{cc|cc|cc|cc|c}
2 & 2 & 2 & 2 & 2 & 2 & 2 & 2 & 1 \\
1 & 1 & 1 & 1 & 1 & 1 & 1 & 1 & 2 \\
\end{array}\right], \]
\[
\left[\begin{array}{cc|cc|cc|cc|c}
1 & 1 & 1 & 1 & 1 & 1 & 1 & 1 & 1 \\
2 & 2 & 2 & 2 & 2 & 2 & 2 & 2 & 2 \\
\end{array}\right],
\left[\begin{array}{cc|cc|cc|cc|c}
1 & 1 & 1 & 1 & 2 & 2 & 2 & 2 & 1 \\
2 & 2 & 2 & 2 & 1 & 1 & 1 & 1 & 2 \\
\end{array}\right].\]}
\end{example}

\section{Counting Invariants}\label{INV}

Let $L$ be a blackboard framed twisted virtual link diagram. 
Recall that a semiarc is a portion of $L$ between classical or virtual 
crossing points or twist bars. Let us define a 
\textit{classical semiarc} as a portion of a twisted virtual knot or link
obtained by dividing at twist bars and classical over and under 
crossing points only; classical semiarcs may contain virtual crossing points.

\begin{definition}\label{def:tvblabel}
\textup{Let $L$ be a twisted virtual link diagram and $X$ a twisted 
virtual birack. A \textit{twisted virtual birack labeling} of $L$
by $X$, or just an $X$\textit{-labeling of} $L$, is an assignment
of an element of $X$ to every semiarc in $L$ such that at every
classical crossing, virtual crossing, and twist bar we have}
\[\includegraphics{jessc-sn2-18.png}\quad \includegraphics{jessc-sn2-19.png}
\quad \includegraphics{jessc-sn2-20.png}\quad \includegraphics{jessc-sn2-21.png}\]
\end{definition}

The twisted virtual birack axioms are consequences of the oriented 
blackboard framed twisted virtual Reidemeister moves using the labeling
conventions in definition \ref{def:tvblabel}. Thus, by construction we have

\begin{theorem}
If $L$ and $L'$ are twisted virtually blackboard framed isotopic twisted 
virtual links and $X$ is a finite twisted virtual birack, then the number
of $X$-labelings of $L$ equals the number of $X$-labelings of $L'$.
\end{theorem}

As with quandle labelings of oriented classical links, rack labelings of 
blackboard framed classical links, etc., an $X$-labeling of a twisted
virtual link diagram $L$ can be understood as a homomorphism $f:TVB(L)\to X$
of twisted virtual biracks where $TVB(L)$ is the \textit{fundamental
twisted virtual birack} of $L$ defined below. More precisely, let $G$ be a set 
of symbols, one for each semiarc in $L$, and define the set of \textit{twisted
virtual birack words in $G$, $W(G)$}, recursively by the rules
\begin{itemize}
\item $g\in G\Rightarrow g\in W(G)$ and
\item $g,h\in W(G)\Rightarrow B_{k}^{\pm 1}(g,h),\ S_{k}^{\pm 1}(g,h),\ 
V_{k}^{\pm 1}(g,h),\ vS_{k}^{\pm 1}(g,h),\ T(g)\in W(G)$ for $k=1,2$.
\end{itemize}
Then the \textit{free twisted virtual birack} on $G$ is the set of equivalence
classes in $W(G)$ modulo the equivalence relation generated by the 
twisted virtual birack axioms, and the \textit{fundamental
twisted virtual birack of $L$} is the set of equivalence classes of elements
of the free twisted virtual birack on $G$ modulo the equivalence relation
generated by the crossing relations. Both sets are twisted virtual biracks
under the operations
\[B([x],[y])=([B_1(x,y)],[B_2(x,y)]),\quad V([x],[y])=([V_1(x,y)],[V_2(x,y)]),
\quad T([x])=[T(x)]\]
where $[x]$ is the equivalence class of $x\in W(G)$.

In \cite{N2} it is observed that the number of labelings by a rank $N$ birack 
of a link diagram is unchanged by the \textit{$N$-phone cord move}:
\[\includegraphics{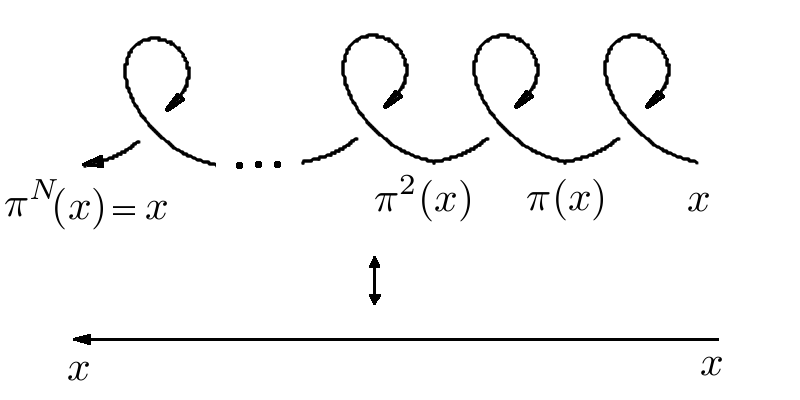}\]

In particular, the number of labelings is periodic in the writhe of each
component of $L$ with period $N$, and consequentially the sum of the
numbers of labelings over a complete period of writhes mod $N$ forms an 
invariant of unframed isotopy. We would like to extend this invariant to
the category of twisted virtual biracks.

\begin{definition}
\textup{Let $L$ be a twisted virtual link with $c$ components 
and $X$ a twisted virtual birack with rank $N$. The \textit{integral 
twisted virtual birack counting invariant} is the number of $X$-labelings 
of $L$ over a complete period of blackboard framings of $L$ mod $N$.
That is,}
\[\Phi_{X}^{\mathbb{Z}}(L)=\sum_{\mathbf{w}\in(\mathbb{Z}_N)^c} 
|\mathrm{Hom}(TVB(L,\mathbf{w}),X)|\]
\textup{where $TVB(L,\mathbf{w})$ is a diagram of $L$ with framing vector
$\mathbf{w}\in(\mathbb{Z}_N)^c$.}
\end{definition}

By construction, we have

\begin{theorem}
If $L$ and $L'$ are twisted virtually isotopic twisted virtual links
and $X$ is a finite twisted virtual birack, then 
$\Phi_{X}^{\mathbb{Z}}(L)=\Phi_{X}^{\mathbb{Z}}(L')$.
\end{theorem}

Starting with a virtual link diagram, it natural to ask which
placements of twist bars yield distinct twisted virtual links. In light of
moves tI and tII, we can place at most one twist bar on any portion
of the knot between classical crossing points, i.e. on any classical 
semiarc.

\begin{example}\textup{
The \textit{virtual Hopf link} is the smallest (in terms of classical and
virtual crossing numbers) nontrivial virtual link with two components. It 
has only two classical semiarcs, so there are $2^2=4$ potentially different 
twisted links which project to the virtual Hopf link under removal of twist 
bars.}

\[\includegraphics{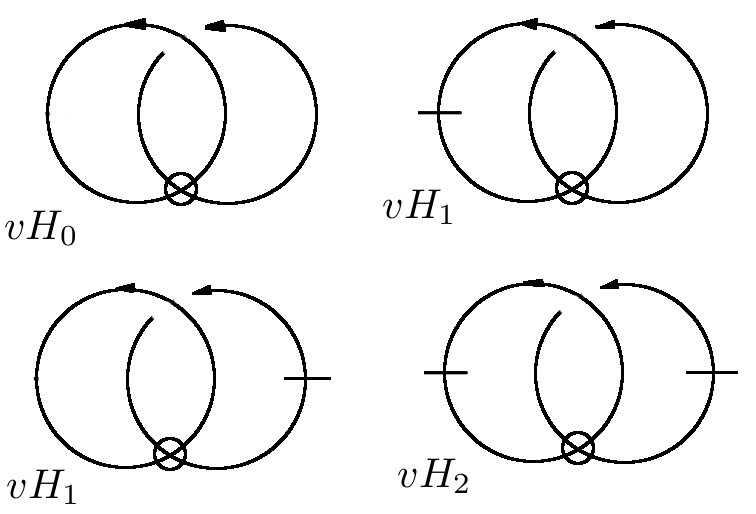} \]

\textup{We can see that two of the possibilities are equivalent using twisted
virtual isotopy moves:}
\[\begin{array}{c}
\includegraphics{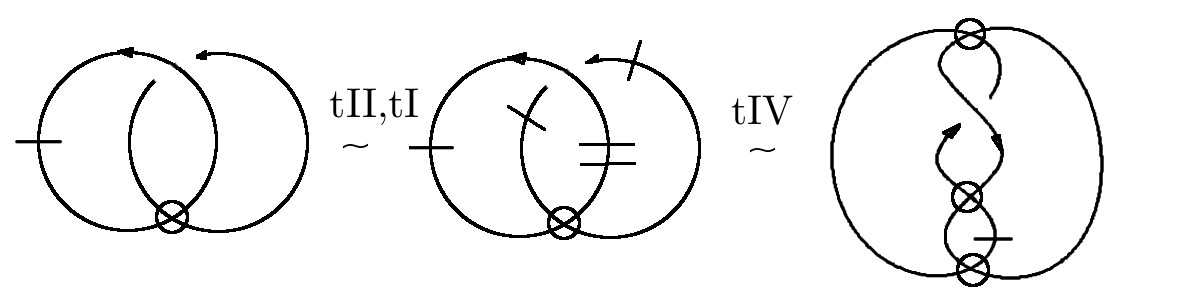} \\
\includegraphics{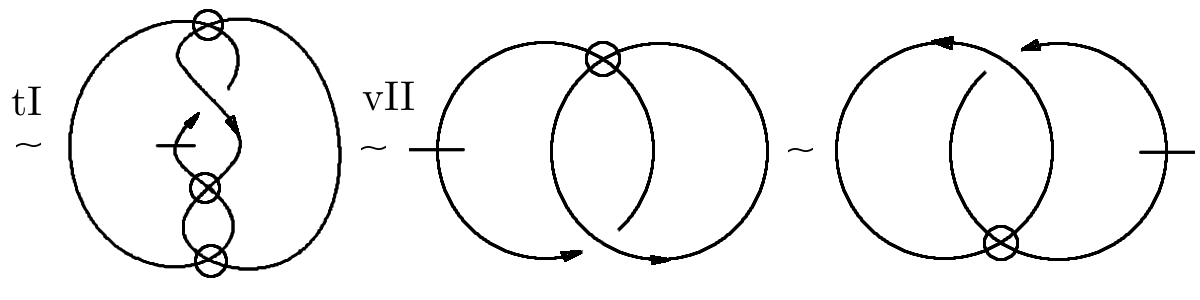} \\
\end{array}\]
\textup{We note that in this move sequence, the twist bar appears to move 
from one component of the link to the other; however, this is an illusion 
due to the symmetry of the link; in fact, the twisted virtual moves have 
changed the over/under relationship of the crossing..}

\textup{On the other hand, we can use the integral counting invariant 
$\Phi_{X}^{\mathbb{Z}}(L)$ with respect to various twisted virtual biracks
to show that the links with two twist bars, one twist bar and zero twist
bars are not equivalent. Let $X_1$ be the twisted virtual birack with
matrix
\[M_{X_1}=\left[\begin{array}{cc|cc|cc|cc|c}
2 & 2 & 2 & 2 & 2 & 2 & 2 & 2 & 2 \\
1 & 1 & 1 & 1 & 1 & 1 & 1 & 1 & 1
\end{array}\right]\]
Then we have $\Phi_{X_1}^{\mathbb{Z}}(vH_0)=4$, $\Phi_{X_1}^{\mathbb{Z}}(vH_1)=0$, 
and $\Phi_{X_1}^{\mathbb{Z}}(vH_2)=0$. Similarly,
let $X_2$ be the twisted virtual birack with
matrix
\[M_{X_2}=\left[\begin{array}{ccc|ccc|ccc|ccc|c}
2 & 2 & 1 & 2 & 2 & 1 & 1 & 1 & 1 & 1 & 1 & 1 & 2 \\
1 & 1 & 2 & 1 & 1 & 2 & 2 & 2 & 2 & 2 & 2 & 2 & 1 \\
3 & 3 & 3 & 3 & 3 & 3 & 3 & 3 & 3 & 3 & 3 & 3 & 3\\
\end{array}\right]\]
Then we have $\Phi_{X_2}^{\mathbb{Z}}(vH_0)=5$, $\Phi_{X_2}^{\mathbb{Z}}(vH_1)=3$, 
and $\Phi_{X_2}^{\mathbb{Z}}(vH_2)=5$.
}\end{example}

Finally, we note that as in the case of biracks, we have several enhancements 
of the twisted virtual birack counting invariant. An \textit{enhancement} 
associates an $X$-labeled move invariant signature to each labeling
of a twisted virtual link diagram so that instead of simply counting 
labelings, we collect the signatures to get a multiset whose cardinality
recovers the counting invariant but is in general a stronger invariant.
We usually convert the multiset into a polynomial for ease of comparison
by taking a generating function. The standard enhancements include the 
following:
\begin{itemize}
\item \textit{Image enhancement.} Given a valid labeling $f:TVB(L)\to X$ of 
the semiarcs in twisted virtual link diagram $L$ of $c$ components by a
twisted virtual birack $X$ of rank $N$, the image of $f$ is an invariant of
twisted virtual isotopy. From a labeled link diagram, we can compute 
$\mathrm{Im}(f)$ by taking the closure under the operations 
$B_1(x,y), B_2(x,y), V_1(x,y), V_2(x,y)$ and $T(x)$ of the set of all
elements of $Y$ appearing as semiarc labels. Then we have an enhanced
invariant
\[\Phi_{X}^{\mathrm{Im}}(L)=\sum_{\mathbf{w}\in(\mathbb{Z}_N)^c}
\left(\sum_{f\in\mathrm{Hom}(TVB(L,\mathbf{w}),X)} u^{|\mathrm{Im}(f)|}\right).\] 
\item \textit{Writhe enhancement.} For this one, we simply keep track of 
which writhe vectors contribute which labelings. For a writhe vector 
$\mathbf{w}=(w_1,\dots,w_c)$, let us denote $q^{\mathbf{w}}=q_1^{w_1}\dots 
q_c^{w_c}$.
Then the \textit{writhe enhanced invariant} is
\[\Phi_{X}^{W}(L)=\sum_{\mathrm{w}\in(\mathbb{Z}_n)^c} 
|\mathrm{Hom}(TVB(L,\mathbf{w}),X)|q^{\mathbf{w}}.\]
\item \textit{Twisted virtual birack polynomials.} Let $X$ be a finite twisted
virtual birack with birack matrix $[M_1|M_2|M_3|M_4|M_5]$. 
For each element $x_k\in X=\{x_1,\dots, x_n\}$, 
let 
\[c_i(x_k)=|\{j\ :\ M_i[x_j,x_k]=x_j\}|\quad \mathrm{and}\quad 
r_i(x_k)=|\{j\ :\ M_i[x_k,x_j]=x_k\}|.\] Then for any twisted 
virtual subbirack $Y\subset X$,
the sub-TVG polynomial of $Y$ is
\[p_{Y\subset X}=\sum_{x\in Y} \left(\sum_{i=1}^5 t_i^{c_i(x)}s_i^{r_i(x)}\right).\]
Then for each $X$-labeling $f$ of $L$, the twisted virtual subbirack 
polynomial of the image of
$f$ gives an invariant signature, so we have the \textit{twisted virtual 
birack polynomial enhanced invariant}
\[\Phi_{X}^{p}(L)=\sum_{\mathbf{w}\in(\mathbb{Z}_N)^c}
\left(\sum_{f\in\mathrm{Hom}(TVB(L,\mathbf{w}),X)} u^{p_{\mathrm{Im}(f)\subset X}}\right).\]
See \cite{N3} for more. 
\end{itemize}

\section{\large\textbf{Questions}}\label{Q}

In remark \ref{noinc} we observed that trivial virtual and twisted
operations do not generally give a birack the structure of a twisted 
virtual birack. In \cite{SPC}, a construction called the 
\textit{twisted product} is given in which a birack $B$ and a choice of
automorphism of $B$ are used to define a twisted virtual birack structure
on the Cartesian product $B\times B$. What other ways are there to define a
twisted virtual birack given a birack?

As with virtual knots, twisted virtual knot theory opens the possibility
of new invariants of classical knots and link defined in terms of twisted
virtual links, since classical links form a subset of twisted virtual links
and any invariant of twisted virtual isotopy is \textit{a fortiori} an 
invariant of classical links. What new invariants of classical and virtual 
links can only be defined using twisted virtual biracks?

We have identified a few enhancements of twisted virtual biracks; what are
some additional enhancements? What is the best way to generalize virtual 
biquandle homology (see \cite{CN}) to define twisted virtual birack homology?

For classical knots, the fundamental quandle is a complete invariant up to
ambient homeomorphism; it is conjectured (see \cite{FJK}) that the fundamental
biquandle is a complete invariant of virtual knots up to vertical mirror image.
Is the fundamental twisted virtual biquandle a complete invariant of twisted
virtual knots up to vertical mirror image? What conditions on 
twisted virtual knots with isomorphic fundamental twisted virtual birack
suffice to guarantee that the knots are twisted-virtually isotopic?

\noindent
\textsc{Department of Mathematical Sciences \\
Claremont McKenna College \\
850 Columbia Ave. \\
Claremont, CA 91711}

\end{document}